\documentclass{amsart}

\usepackage[mathcal]{eucal}
\usepackage{amssymb}
\usepackage{graphicx}
\usepackage[all]{xy}
\usepackage[breaklinks=true,pagebackref=true]{hyperref}

\newtheorem{thm}{Theorem}
\newtheorem{prop}[thm]{Proposition}
\newtheorem{lem}[thm]{Lemma}

\theoremstyle{remark}
\newtheorem{rem}[thm]{Remark}

\theoremstyle{definition}

\newcommand{\C}{\mathbb C}
\newcommand{\R}{\mathbb R}

\newcommand{\HH}{\mathbb H}

\newcommand{\T}{\mathcal T}

\newcommand{\I}{\mathrm{I}}
\newcommand{\id}{\mathrm{id}}
\newcommand{\la}{\lambda}
\newcommand{\g}{\gamma}

\newcommand{\s}{\sigma}

\newcommand{\MCG}{\mathrm{MCG}}
\newcommand{\pt}{\mathrm{Pt}}
\newcommand{\sh}[1]{\mathrm{Sh}(\T(#1))}
\newcommand{\wt}[1]{\widetilde{#1}}

\newcommand{\PPt}{\mathfrak{Pt}}

\renewcommand{\phi}{\varphi}

\title[]
{Ptolemy groupoids, shear coordinates and the augmented Teichm\"uller space}
\author{Julien Roger}
\address{Department of Mathematics, Rutgers University, New Brunswick NJ~08854}
\email{juroger@math.rutgers.edu}
\urladdr{math.rutgers.edu/$\sim$juroger}
\thanks{This research was partially supported by the grant DMS-1207832 from the National Science Foundation.}

\begin{document}

\begin{abstract}
We start by describing how ideal triangulations on a surface degenerate under pinching of a multicurve. We use this process to construct a homomorphism from the Ptolemy groupoid of a surface to that of a pinched surface which is natural with respect to the action of the mapping class group. We then apply this construction to the study of shear coordinates and their extension to the augmented Teichm\"uller space. In particular, we give an explicit description of the action of the mapping class group on the augmented Teichm\"uller space in terms of shear coordinates.
\end{abstract}

\maketitle


An \emph{ideal triangulation} $\la$ of a punctured surface $S$ consists of a maximal collection of isotopy classes of arcs connecting the punctures of $S$, which decompose the surface into (ideal) triangles. The goal of this paper is to study how ideal triangulations behave under pinching of curves on surfaces and describe some applications to the study of Teichm\"uller space, the Ptolemy groupoid and the action of the mapping class group on them.

Given a multicurve $\g=\cup^l_{i=1}\g_i$ in $S$, that is, a disjoint union of simple closed curves which are not homotopic to each other, not null-homotopic or homotopic to a puncture, we consider the subsurface $S_\g=S\smallsetminus\g$, which can be thought of as being obtained from $S$ by pinching the components of $\g$ to points and removing them. The key observation is the following: given an ideal triangulation $\la$ of $S$, there is a natural way of obtaining an ideal triangulation of $S_\g$, which we call the \emph{induced ideal triangulation} $\la_\g$. In practice, $\la_\g$ is obtained from $\la$ by restricting its edges to $S_\g$ and grouping the segments obtained into isotopy classes (see Figure~\ref{example2} for an example).

\begin{figure}[htb!]
\includegraphics{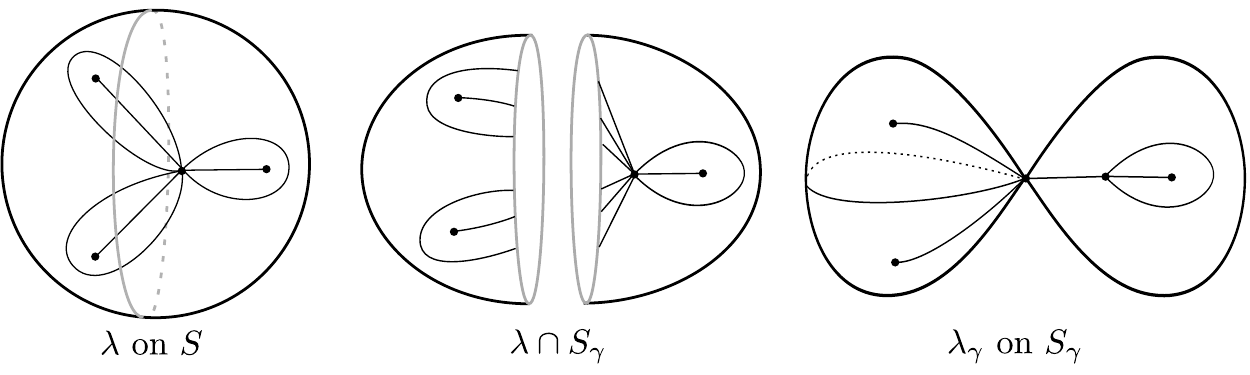}
\caption{\label{example2} An example on the four punctured sphere}
\end{figure}
This construction is in essence a 2--dimensinal version of the crushing of ideal triangulations of 3--manifolds along normal surfaces introduced by W. Jaco and J. Rubinstein in \cite{JaRu}.

In the first part of this paper we describe how this process translates in terms of the \emph{Ptolemy groupoid} $\pt(S)$ \cite{Penner1} with set of objects given by ideal triangulations and generators given by \emph{diagonal exchanges} (see Figure~\ref{diagexch}).

\begin{thm}
\label{intro morph}
Given any multicurve $\g$ on $S$, the assignment $\la\mapsto\la_\g$ extends naturally to a homomorphism of groupoids
\[\pi_\g\colon\pt(S)\to\pt(S_\g).\]
\end{thm}

The (pure) mapping class group $\MCG(S)$ acts on the set of ideal triangulations so that, for any ideal triangulation $\la$, there is a natural homomorphism $\Phi_\la\colon\MCG(S)\to\pt(S)$. Pinching in terms of mapping class groups can be understood via the \emph{cutting homomorphism} $\zeta$ and the associated exact sequence
\[1\to<T_{\g_1},\ldots,T_{\g_l}>\to\MCG(S,\g)\stackrel{\zeta}{\rightarrow}\MCG(S_\g)\]
where the  $T_{\g_i}$'s denote the Dehn twists along the components of $\g$. Here $\MCG(S,\g)$ denotes the stabilizer of $\g$ in $\MCG(S)$ and $\zeta$ associates to each mapping class in the stabilizer its restriction to $S_\g$. As a direct consequence of Theorem~\ref{intro morph}, we obtain the following result.
\begin{thm}
The following diagram of exact sequences is commutative.
\begin{align*}
\label{diagram}
\xymatrix{1\ar[r]&{<T_{\g_1},\ldots,T_{\g_n}>}\ar[d]_{\Phi_{\la|}}\ar[r]&{\MCG(S,\g)}\ar[d]_{\Phi_\la}\ar[r]^{\zeta}&{\MCG(S_\g)}\ar[d]^{\Phi_{\la_\g}}\\
1\ar[r]&{\ker\pi_\g}\ar[r]&{\pt(S)}\ar[r]_{\pi_\g}&{\pt(S_\g)}
}
\end{align*}
\end{thm}

In the second part of this paper, we apply this construction to the study of the \emph{augmented Teichm\"uller space} $\overline{\T(S)}$ \cite{Abi1,Abi2,Ber1} and the action of the mapping class group on it. As a set, this augmentation is obtained by adjoining to $\T(S)$ the Teichm\"uller spaces $\T(S_\g)$ of the pinched surfaces $S_\g$ for all the multicurves $\g$. It admits a natural topology for which the action of the mapping class group on Teichm\"uller space extends continuously \cite{Abi2}. The quotient of $\overline{\T(S)}$ by this action is the Deligne-Mumford compactification of the moduli space.

Given an ideal triangulation $\la$, we consider the associated \emph{(exponential) shear coordinates} introduced by W. Thurston \cite{Thurston} and studied further by F. Bonahon \cite{Bo1}, which provide an embedding $x_\la\colon\T(S)\to\R^\la$. For any multicurve $\g$, we can also consider the shear coordinates $y_{\la_\g}\colon\T(S_\g)\to\R^{\la_\g}$ on the strata $\T(S_\g)$ and ask if they are extensions of the coordinates $x_\la$ for the topology of the augmented Teichm\"uller space.

The key ingredient is given by the map
\[\Theta_{\g,\la}\colon\R^\la\to\R^{\la_\g}\]
which associates to each edge of $\la_\g$ the products of shear parameters coming from the segments of edges of $\la$ forming its homotopy class. This map was introduced in its quantized version in \cite{moi} where we showed that it respects the Poisson structure given by the relevant Weil-Petersson bivectors \cite{fo1}. We use this map to describe how the topology of the augmented Teichm\"uller space can be recovered using shear coordinates and how the charts $x_\la$ and $y_{\la_\g}$ are related.

\begin{thm}
\label{intro ext}
Let $m_t\in\T(S)$, $t>0$, be a continuous family of hyperbolic metrics on $S$ and let $m_\g\in\T(S_\g)$. Suppose that $\lim_{t\to 0}\ell_{m_t}(\g_i)=0$ for each component of $\g$.
Then
\[m_t\xrightarrow[t\rightarrow 0]{} m_\g\text{ in }\overline{\T(S)}\Longleftrightarrow \lim_{t\to 0}\Theta_{\g,\la}(x_\la(m_t))=y_{\la_\g}(m_\g).
\]
\end{thm}

The forward implication was proved in \cite{moi} and we provide a proof of the converse in Appendix~\ref{shearext}. We note that this description of the topology of the augmented Teichm\"uller space was used recently by D. \v Sari\'c \cite{Saric} to obtain a hyperbolic geometric proof  of H. Masur's result \cite{Masur1} about the extensions of the Weil-Petersson metric to this augmentation.

A key fact about this construction is that the choice of a single ideal triangulation provides coordinates on all the lower-dimensional strata at once. In this way, we can study the action of the mapping class group on the whole augmentation using a single set of coordinates and its natural extensions. First, given two ideal triangulations $\la$ and $\la'$ of $S$ and the associated change of coordinates $\Phi_{\la\la'}=x_\la\circ x^{-1}_{\la'}\colon\R^{\la'}\to\R^{\la}$ on $\T(S)$, we have an induced change of coordinates $\Phi_{\la_\g\la'_\g}=y_{\la_\g}\circ y_{\la'_\g}\colon\R^{\la'_\g}\to\R^{\la_\g}$ on $\T(S_\g)$ (these changes of coordinates are in fact only defined on some open dense subset of their domain which contains the image of $\T(S)$ as a closed subset). This assignment is well-defined and respects compositions. As a direct consequence of Theorem~\ref{intro ext}, we have the following proposition.

\begin{prop}
When approaching a stratum $\T(S_\g)$ from $\T(S)$, the composition $\Phi_{\la\la'}\circ\Theta_{\g,\la}$ tends to the change of coordinates $\Phi_{\la_\g\la'_\g}$.
\end{prop}

An important aspect of shear coordinates associated to ideal triangulations is that the corresponding changes of coordinates $\Phi_{\la\la'}$ are rational. As such, one can construct a natural representation of $\MCG(S)$ as rational functions on $\R^\la$ via
\[f\mapsto \Phi_{f,\la}=\I_{f,\la}\circ\Phi_{f(\la)\la}\]
where $\I_{f,\la}\colon\R^{f(\la)}\to\R^\la$ is induced by the natural identification of the edges of $\la$ and $f(\la)$. This construction was described first by R. Penner \cite{Penner1} using the closely related $\lambda$--lengths coordinates on the decorated Teichm\"uller space.

\begin{thm}
The action of $\MCG(S)$ on $\overline{\T(S)}$ is given in shear coordinates by a collection of rational maps
\[\Phi_{f,\la,\g}\colon\R^{\la_\g}\to\R^{\la_{f^{-1}(\g)}}\]
for all multicurves $\g$, induced by the changes of coordinates $\Phi_{f(\la)_\g\la_\g}$. These maps coincide with the action of $\MCG(S_\g)$ on $\T(S_\g)$ in the sense that
\[\Phi_{f,\la,\g}=\Phi_{\zeta(f),\la_\g}\]
for the cutting homomorphism $\zeta\colon\MCG(S,\g)\to\MCG(S_\g)$.
\end{thm}

Some of the ideas behind this article can be found in \cite{moi}. In that paper, they were used to study the behavior of the irreducible representations of the quantum Teichm\"uller space $\T^q(S)$ under pinching of multicurves. Irreducible representations of $\T^q(S)$ can be used to construct a projective vector bundle over the classical Teichm\"uller space which descends to the moduli space (see \cite{BBL} for its construction in a related setting). Using the results from \cite {moi} and the present article, we expect to show in an upcoming paper \cite{moi2} that this bundle extends to the augmented Teichm\"uller space and descends to the Deligne-Mumford compactification of the moduli space. More speculatively, there is a possibility that this could lead to the construction of a modular functor, that is, a family of finite dimensional representations of the mapping class groups which are natural under pinching.

\section{Ptolemy groupoids and mapping class group action}

\subsection{Ideal triangulations and Ptolemy groupoids}

Throughout this article $S$ will denote a closed connected oriented surface of genus $g$ with $n$ punctures such that $n\geq 1$ and $\chi(S)<0$. A $\emph{multicurve}$ $\g=\cup^l_{i=1}\g_i$ is the isotopy class of a finite union of simple closed curves $\g_i$ which are not homotopic to each other, not null-homotopic or homotopic to a puncture. We will denote by $S_\g=S\smallsetminus\g$ the surface obtained by removing from $S$ (a representative of) the multicurve $\g$. Once again, such a surface will be considered up to isotopy of $S$. Topologically, $S_\g$ is a possibly disconnected surface with two new punctures for each component of $\g$ removed. However, it is important to remember that $S_\g$ is a subsurface of $S$. As such,  if $f$ is a diffeomorphism of $S$ that is not isotopic to the identity, the surfaces $S_\g$ and $f(S_\g)=S_{f(\g)}$, though topologically equivalent, are not identified. Note also that the new punctures coming from the same component of $\g$ are implicitely paired. Finally, we 
consider the \emph{(pure) mapping class group} $\MCG(S)$ consisting of isotopy classes of orientation-preserving diffeomorphisms of $S$ which fix the punctures. With this definition, if $S_\g$ is disconnected, the elements of $\MCG(S_\g)$ cannot interchange its connected components.

An \emph{ideal triangulation} $\lambda$ of $S$ consists of a maximal collection of isotopy classes of disjoint arcs between the punctures of $S$ decomposing the surface into ideal triangles. Given our requirements for $S$, such triangulations always exist and we denote by $\Lambda(S)$ the set of ideal triangulations of $S$.

If two triangulations $\la$ and $\la'$ are identical except for one edge in a square as in Figure~\ref{diagexch}, we say that they differ by a \emph{diagonal exchange}. Note that some of the sides of the square can be identified.
\begin{figure}[htb!]
\includegraphics{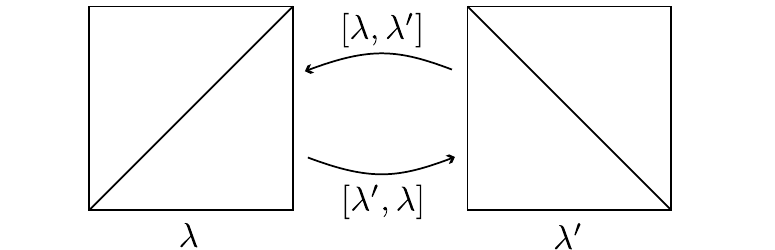}
\caption{\label{diagexch}Bigon Relation}
\end{figure}

The combinatorics of diagonal exchanges between ideal triangulations was studied first by Mosher \cite{Mo1} and analized further by Harer \cite{Har1} and Penner \cite{Penner1}. The following terminology and algebraic definition were introduced by Penner. We define the Ptolemy groupoid $\pt(S)$ over the set of objects $\Lambda(S)$ as follows: the generators are given by diagonal exchanges and we denote by $[\la,\la']$ the generator associated to a diagonal exchange from $\la'$ to $\la$ as in Figure~\ref{diagexch}. In addition, For each triangulation $\la$ we have an identity element, denoted $\id_\la$. The relations between generators are of three types:
\begin{itemize}
\item \emph{Bigon relation:} If $\la$ and $\la'$ differ by a diagonal exchange as in Figure~\ref{diagexch}, then 
\[[\la,\la'][\la',\la]=\id_\la;\]

\item \emph{Square relation:} If $\la_1$, $\la_2$, $\la_3$ and $\la_4$ differ by diagonal exchanges in non-overlapping squares as in Figure~\ref{srel}, then
\[[\la_1,\la_2][\la_2,\la_3][\la_3,\la_4][\la_4,\la_1]=\id_{\la_1};\]

\item \emph{Pentagon relation:} If $\la_1$ through $\la_5$ differ by diagonal exchanges along two overlapping squares forming a pentagon as in Figure~\ref{prel}, then
\[[\la_1,\la_2][\la_2,\la_3][\la_3,\la_4][\la_4,\la_5][\la_5,\la_1]=\id_{\la_1}.\]
\end{itemize}
Note that, thanks to the bigon relation, we will not specify orientations of edges in the figures describing relations in the Ptolemy groupoid. 
\begin{figure}[htb!]
\includegraphics{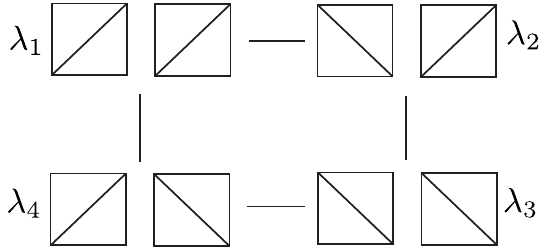}
\caption{\label{srel}Square Relation}
\end{figure}

\begin{figure}[htb!]
\includegraphics{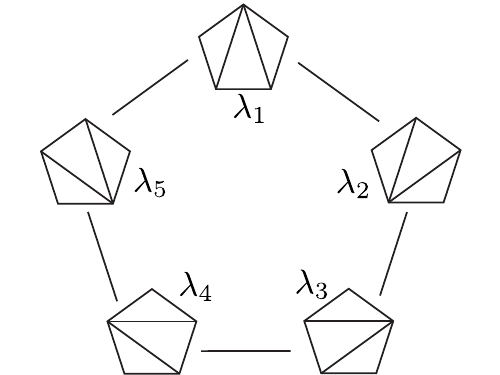}
\caption{\label{prel}Pentagon Relation}
\end{figure}

The following facts about the Ptolemy groupoid are key \cite{Har1,Penner1}:
\begin{itemize}
\item given any two ideal triangulations $\la,\la'\in\Lambda(S)$ there exists a sequence $\la'=\la_1$, $\la_2$, \ldots, $\la_n=\la$ of triangulations, such that, for any $i$, $\la_i$ and $\la_{i+1}$ differ by a diagonal exchange. We call such a sequence a \emph{path from $\la'$ to $\la$};
\item any two such sequences are equivalent modulo the relations in the Ptolemy groupoid. That is, if $\mu_1$ ,\ldots, $\mu_m$ is another path from $\la'$ to $\la$, we have
\[
[\mu_1,\mu_2][\mu_2,\mu_3]\cdots[\mu_{m-1},\mu_m]=[\la_1,\la_2][\la_2,\la_3]\cdots[\la_{n-1},\la_n]
\]
\end{itemize}

We denote by $[\la,\la']$ the element corresponding to the product of generators forming any path from $\la'$ to $\la$. The second property implies that these elements satisfy the composition relation $[\la,\la'][\la'\la'']=[\la,\la'']$. In addition, we have $[\la,\la]=\id_\la$. We use ``contravariant'' conventions since these will be related to changes of coordinates later on.

Alternatively, one can define the Ptolemy groupoid in terms of the associated CW--complex $\PPt(S)$ with vertex set $\Lambda(S)$, edges given by elementary moves and 2--cells given by the square and pentagon relations. The Ptolemy groupoid is then identified with the fundamental groupoid of $\PPt(S)$ with base points $\Lambda(S)$. In this description, the facts above correspond to $\PPt(S)$ being connected and simply connected respectively.

\subsection{Pinching surfaces and induced ideal triangulations}

We begin by recalling a construction described in \cite{moi} associating to each triangulation of $S$ a triangulation of $S_\g$.

Given an ideal triangulation $\la$ of $S$ we define the associated \emph{induced ideal triangulation} $\la_\g$ of $S_\g$ as follows: Choose a representative of $\g$ which minimizes intersections with $\la$ and denote by $\la\cap S_\g$ the family of arcs obtained from the edges of $\la$ restricted to $S_\g$. The edges of the induced triangulation $\la_\g$ are then obtained by grouping the segments in $\la\cap S_\g$ into distinct isotopy classes of arcs (see Figure~\ref{example2} for an example). Lemma~4 in \cite{moi} asserts that this is indeed an ideal triangulation. In practice, $\la_\g$ is obtained from $\la$ by collapsing each bigon in the decomposition of $S_\g$ by $\la\cap S_\g$ to arcs (see Figure~\ref{induced}). Note that, since $S$ and $S_\g$ have the same Euler characteristic, $\la$ and $\la_\g$ have the same number of edges.

This process defines a map between sets of ideal triangulations
\begin{align*}
\label{map}
\pi_\g\colon\Lambda(S)&\to\Lambda(S_\g)\\
              \la&\mapsto\la_\g
\end{align*}
which is our first object of study.

\begin{rem}
\label{triangles}
In this process, one can see that ideal triangles for $\la_\g$ on $S_\g$ are identified naturally with ideal triangles for $\la$ on $S$ (see Figure~\ref{induced}).
\begin{figure}[htb!]
\includegraphics{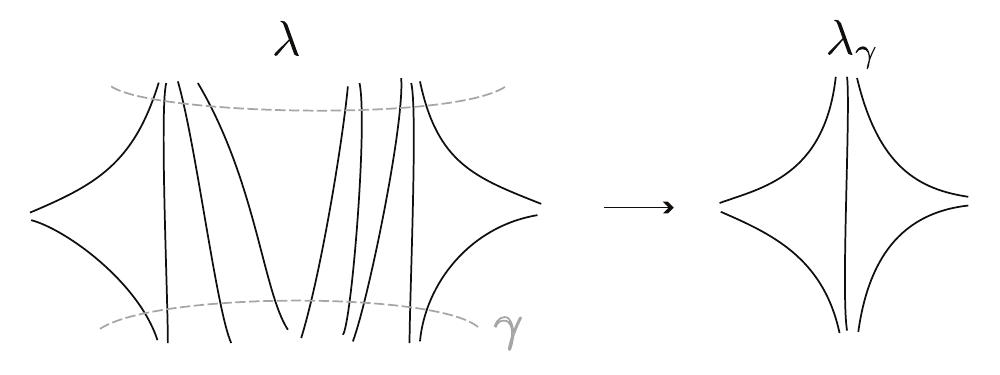}
\caption{\label{induced} Triangles before and after}
\end{figure}
\end{rem}


\subsection{A homomorphism between Ptolemy groupoids}

In the last section we defined a map $\pi_\g$ between the sets of objects of the Ptolemy groupoids $\pt(S)$ and $\pt(S_\g)$. To extend it to a map of groupoids, we start by describing the effect of pinching on diagonal exchanges.

Let $Q$ be a square (that is, the union of two distinct adjacent triangles) in an ideal triangulation $\la$. We say that a multicurve $\g$ \emph{crosses} $Q$ if at least one of its components crosses successively two parallel sides of $Q$, as in the top part of Figure~\ref{crossing}. The lower part of the figure describes the (lack of) effect that a diagonal exchange has on the pinched triangulation. Here $Q_\g$ denotes the two triangles in the pinched surface coming from those forming $Q$. The following lemma is elementary.
\begin{figure}[htb!]
\includegraphics{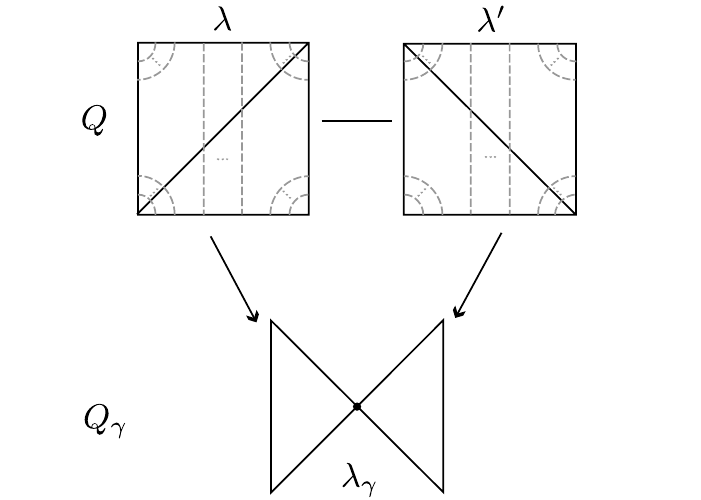}
\caption{\label{crossing} Pinching a square}
\end{figure}

\begin{lem}
\label{diagkilling}
Suppose that $\la$ and $\la'$ differ by a diagonal exchange in a square $Q$. Then
\begin{itemize}
\item if $\g$ does not cross $Q$, $\la_\g$ and $\la'_\g$ differ by a diagonal exchange;
\item if $\g$ crosses $Q$, $\la_\g$ and $\la'_\g$ are equal.
\end{itemize}
\end{lem}

Note that whether $\g$ passes through $Q$ around the corners or whether it crosses it once or multiple times has no effect on the conclusions of the lemma.

\begin{thm}
\label{homom}
The map $\pi_\g\colon\pt(S)\to\pt(S_\g)$ defined on the objects by
\[\pi_\g(\la)=\la_\g\]
and on the generators by
\[\pi_\g([\la,\la'])=[\la_\g,\la'_\g]\]
extends to a homomorphism of groupoids.
\end{thm}
\begin{proof}
It suffices to check that $\pi_\g$ sends relations in $\pt(S)$ to relations in $\pt(S_\g)$.

Notice first that if $\la_\g=\la'_\g$ then $\pi_\g([\la,\la'])=[\la_\g,\la_\g]=\id_{\la_\g}$. By Lemma~\ref{diagkilling} this happens if and only if $\g$ crosses the square in which the diagonal exchange is performed. Hence a bigon relation is either sent to a bigon relation or to an identity map.

For the square relation we note that, by Remark~\ref{triangles}, if two squares do not overlap then, after pinching, the triangles forming them are still distinct. Then, by Lemma~\ref{diagkilling}, distant diagonal exchanges are sent either to distant diagonal exchanges or to a combination of diagonal exchanges and identity maps. Namely, if $\g$ doesn't cross either square, then the square relation is sent to a square relation, if $\g$ crosses exactly one square, it is sent to a bigon relation and if it crosses both squares, it is sent to an identity map.

Finally, for the pentagon relation, we let $\mathcal{P}$ be the pentagon in which the triangulations $\la_1$ through $\la_5$ differ as in Figure~\ref{pentagonpinching}. Remember that $\mathcal{P}$ consists of two overlapping squares, hence there are three possibilities depending on the number of squares  of $\mathcal{P}$ that $\g$ crosses. Note that this number is independent of diagonal exchanges performed inside $\mathcal{P}$. If $\g$ doesn't cross any square, then the pentagon relation is sent to the corresponding pentagon relation in $\pt(S_\g)$. If $\g$ crosses exactly one of the squares in $\mathcal{P}$, we are in the situation of Figure~\ref{pentagonpinching}.
\begin{figure}[htb!]
\includegraphics{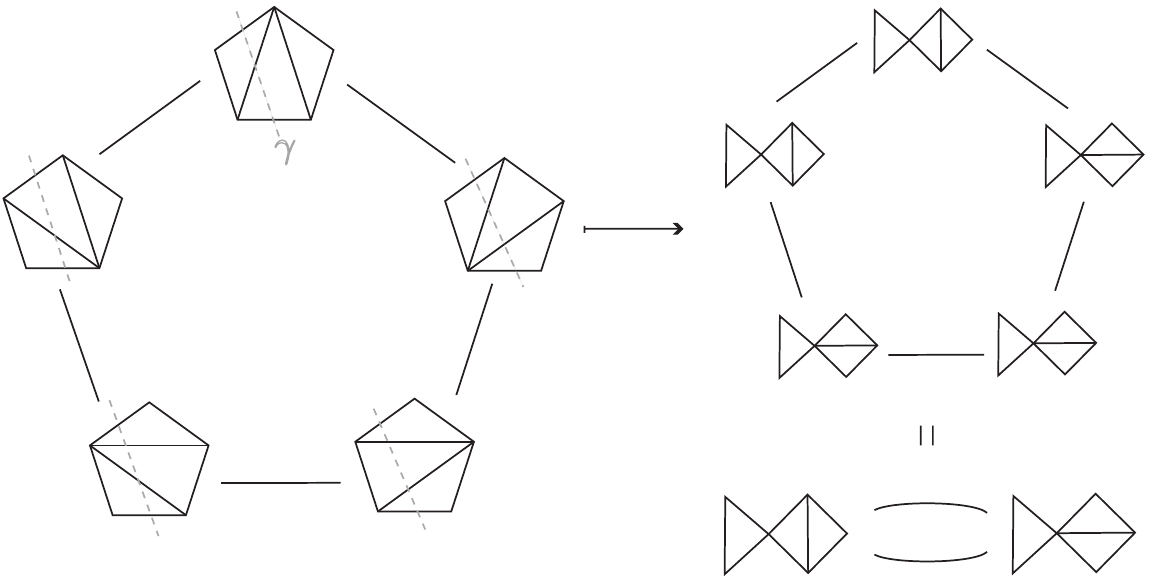}
\caption{\label{pentagonpinching} Pinching a pentagon}
\end{figure}
We see that three of the edges are sent to an identity, while the other two are sent to a diagonal exchange in a square and its inverse. Hence the pentagon relation is sent to a bigon relation in $\pt(S_\g)$. Finally, if $\g$ crosses both squares in $\mathcal{P}$, then the induced triangulations $\pi_\g(\la_i)$ are all identical, and, as such, the pentagon relation is sent to an identity in that case.
\end{proof}

Note that the image of $\pi_\g$ is connected, that is, given $\la_\g$ and $\la'_\g$, there is a path in the image connecting them, namely $[\la_\g,\la'_\g]=\pi_\g([\la,\la'])$.

The homomorphism $\pi_\g$ can be promoted to a cellular map $\Pi_\g\colon \PPt(S)\to \PPt(S_\g)$ as follows: $\Pi_\g$ sends the vertices $\la$ to the vertices $\la_\g$; it sends an edge $[\la,\la']$ either to the edge $[\la_\g,\la'_\g]$ if the triangulations are distinct or to the vertex $\la_\g=\la'_\g$; faces are sent either to the corresponding face, collapsed onto an edge or collapsed onto a point depending on the configurations described in the proof of Theorem~\ref{homom}. This can easily be done in such a way that $\Pi_\g$ is continuous. By construction, we have that $\pi_\g=(\Pi_\g)_*$, the induced map between fundamental groupoids.


\subsection{Mapping class group action}
\label{action}
The mapping class group acts naturally on the set of ideal triangulations via
\begin{align*}
\MCG(S)\times\Lambda(S)&\to\Lambda(S)\\
(f,\la)&\mapsto f(\la).
\end{align*}
Given any $\la\in\Lambda(S)$, this action induces a map
\begin{align*}
\Phi_\la\colon\MCG(S)&\to\pt(S)\\
f&\mapsto[f(\la),\la].
\end{align*}
This map is a homomorphism of groupoids in the sense that
\[\Phi_\la(g\circ f)=\Phi_{f(\la)}(g)\cdot\Phi_\la(f)\]
for any $f,g\in\MCG(S)$.

If we compose this map with the homomorphism described in Theorem~\ref{homom}, we obtain a homomorphism
\[\pi_\g\circ\Phi_\la\colon\MCG(S)\to\pt(S_\g)\]
which we want to investigate further. Note that this map is not obtained, at least in an obvious way, by an action on $\Lambda(S_\g)$.

We denote by $\MCG(S,\g)$ the stabilizer of the components of $\g$ in $\MCG(S)$. That is, given (a representative of) $f\in\MCG(S,\g)$, we ask that $f(\g_i)$ and $\g_i$ be homotopic for each (representatives of) components of $\g$. Note that $f$ and the $\g_i$'s can be chosen so that $f$ fixes $\g$ pointwise.

We then have an exact sequence
\begin{equation}
\label{exactseq}
1\to<T_{\g_1},\ldots,T_{\g_l}>\to\MCG(S,\g)\stackrel{\zeta}{\rightarrow}\MCG(S_\g)\to 1
\end{equation}
where $T_{\g_i}$ denote the Dehn twist along the components $\g_i$ of $\g$ and $\zeta$ is the natural \emph{cutting homomorphism} defined by choosing a representative fixing $\g$ and taking its restriction to $S_\g$ (see for example \cite{mcg} Proposition~3.20). One can also easily see that this map is surjective.

We recall that the \emph{kernel} of a groupoid homomorphism is the subgroupoid consisting of all the morphisms sent to identity elements. In the case at hand this means
\[\ker \pi_\g=\left\{[\lambda,\lambda']\in\pt(S)\,|\,\lambda_\g=\lambda'_\g\right\}.\]
In particular, it contains all the diagonal exchanges inside squares crossed by $\g$.

\begin{thm}
\label{tower}
The following diagram of exact sequences is commutative
\begin{equation}
\label{diagram}
\begin{gathered}
\xymatrix{1\ar[r]&{<T_{\g_1},\ldots,T_{\g_n}>}\ar[d]_{\Phi_{\la|}}\ar[r]&{\MCG(S,\g)}\ar[d]_{\Phi_\la}\ar[r]^{\zeta}&{\MCG(S_\g)}\ar[d]^{\Phi_{\la_\g}}\\
1\ar[r]&{\ker\pi_\g}\ar[r]&{\pt(S)}\ar[r]_{\pi_\g}&{\pt(S_\g)}
}
\end{gathered}
\end{equation}
\end{thm}
\begin{proof}
We first show that $\pi_\g\circ\Phi_\la=\Phi_{\la_\g}\circ\zeta$.

Let $f\in\MCG(S,\g)$, which we identify with a representative fixing $\g$ so that $\zeta(f)=f_{|S_\g}$. We have
\[\pi_\g\circ\Phi_\la(f)=[f(\la)_\g,\la_\g]\]
and
\[\Phi_{\la_\g}\circ\zeta(f)=[\zeta(f)(\la_\g),\la_\g],\]
hence it suffices to show that $f(\la)_\g$ and $\zeta(f)(\la_\g)$ are isotopic.

By definition, $\zeta(f)(\la\cap S_\g)=f_{|S_\g}(\la\cap S_\g)=f(\la\cap S_\g)=f(\la)\cap S_\g$. Recalling that $\la_\g$ is obtained by taking isotopy classes of segments in $\la\cap S_\g$ and similarly for $f(\la)_\g$, we obtain the result.

It remains to show that Dehn twists along components of $\g$ are in the kernel of $\pi_\g$. This follows from the fact that $\la\cap S_\g$ and $T_{\g_i}(\la)\cap S_\g$ can easily be seen to be isotopic in $S_\g$.
\end{proof}

\begin{rem}
The map $\pi_\g$ is not in general surjective. However, using Theorem~\ref{diagram} and the surjectivity of the map $\zeta$, we see that the action of the mapping class group of the subsurface $S_\g$ on $\pt(S_\g)$ is induced by the action of the mapping class group of $S$ on $\pt(S)$. Explicitely, given $g\in\MCG(S_\g)$, pick $f\in\MCG(S,\g)$ such that $\zeta(f)=g$, then
\[\Phi_{\la_\g}(g)=\pi_\g\circ\Phi_\la(f).\]
\end{rem}

\section{Shear coordinates and the augmented Teichm\"uller space}
\label{section 2}

\subsection{Teichm\"uller space and shear coordinates}

An $S$--\emph{marked hyperbolic surface} is an isotopy class of homeomorphisms $[\phi\colon S\to\Sigma]$ where $\Sigma$ is a complete hyperbolic surface with finite area (and hence cusp ends). Two $S$--marked hyperbolic surfaces $[\phi\colon S\to\Sigma]$ and $[\phi'\colon S\to\Sigma']$ are called \emph{equivalent} if there exists an isometry $h\colon\Sigma\to\Sigma'$ such that $h\circ\phi$ is isotopic to $\phi'$. The \emph{Teichm\"uller space} $\T(S)$ is the set of equivalence classes of $S$--marked hyperbolic surfaces. We will often denote an element of $\T(S)$ by the hyperbolic metric $m$ on $S$, the pullback of the one on $\Sigma$ by $\phi$.

The (anti-) action of the mapping class group on the Teichm\"uller space is defined in terms of change of marking: for $f\in\MCG(S)$ and $[\phi\colon S\to\Sigma]\in\T(S)$, the action is given by
\[f^*[\phi\colon S\to\Sigma]=[\phi\circ f\colon S\to\Sigma].\]
In terms of metrics, this action can be understood as the pullback of $m$ by $f$ which we will denote $f^*m$. The quotient of the Teichm\"uller space by this action is the moduli space $\mathcal{M}(S)$.

We will use the notion of shear parameters associated to the edges of an ideal triangulation (or more generally a lamination) introduced by Thurston in \cite{Thurston} and studied further by Bonahon \cite{Bo1}. We will in fact work with the exponentials of these parameters and a good reference for our purpose is \cite{Liu1}. They are obtained as follows: given a hyperbolic metric $m$ on $S$, lift $\la$ to the universal cover $\HH^2$ and consider an edge $\widetilde{e}$ of the lift which projects onto $e\in\la$. The \emph{logarithmic shear parameter}  along $e$ for the metric $m$ is given by the signed distance between the orthogonal projections of the vertices facing $\widetilde{e}$. The distance is measured positively ``to the left'' when facing the edge (see Figure~\ref{shearparam}). 
\begin{figure}[htb!]
\includegraphics{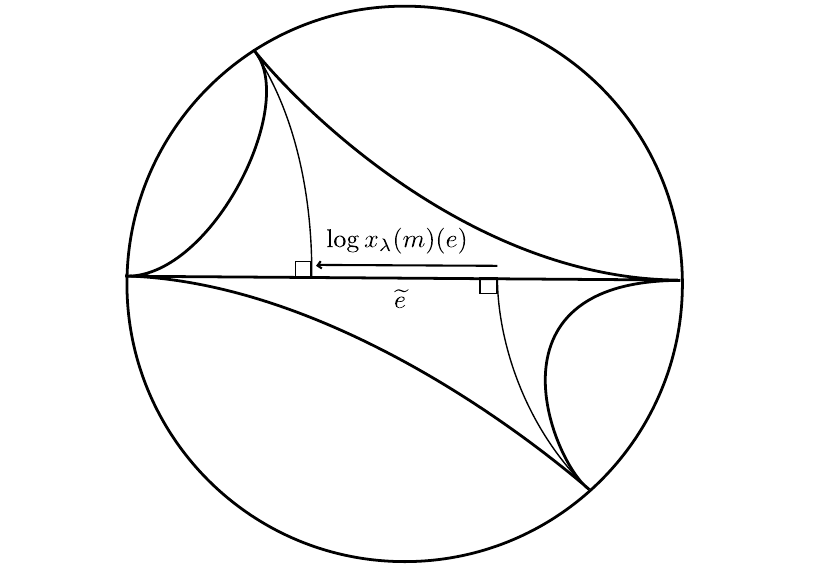}
\caption{\label{shearparam} Geometric interpretation of shear parameters}
\end{figure}
The collection of (exponential) shear parameters  associated to an ideal triangulation $\la$ gives an embedding $x_\la\colon\T(S)\to\R^\la$ which we will call a \emph{chart} by abuse of language. Since we do not assume that triangulations carry labelings, we want to think of $x_\la(m):\la\to\R$ as a function on the edges of $\la$, associating to $e\in\la$ the shear parameter $x_\la(m)(e)$ along $e$ for the metric $m$. These parameters are in fact bona fide coordinates on a larger space (called the enhanced Teichm\"uller space) in which $\T(S)$ embeds.

Given another ideal triangulation $\la'$, we will denote the change of coordinates map by $\Phi_{\la\la'}=x_\la\circ x^{-1}_{\la'}\colon\R^{\la'}\to\R^{\la}$. Here, again by abuse of notation, it should be understood that the map is only defined on some subset of $\R^{\la'}$ (which is in fact open and dense). These maps turn out to be rational, explicit formulae can be found for example in \cite{Liu1} (see also Appendix~\ref{example})

Finally, we consider the \emph{shear coordinates groupoid} $\sh{S}$ with set of objects $\Lambda(S)$ and morphisms given by the changes of coordinates $\Phi_{\la\la'}$ subject to the natural composition relation $\Phi_{\la\la'}\circ\Phi_{\la'\la''}=\Phi_{\la\la''}$. By construction, there is a natural isomorphism of groupoids $\Phi\colon\pt(S)\to\sh{S}$ sending a path $[\la,\la']$ to the associated change of coordinates $\Phi_{\la\la'}$. We can push the map $\pi_\g$ to a map between shear coordinates groupoids so that the following diagram commutes:
\begin{align}
\label{sheardiagram}
\xymatrix{{\pt(S)}\ar[d]^{\pi_\g} \ar[r]^{\Phi} & {\sh{S}}\ar[d]_{\pi_\g}\\
					{\pt(S_\g)} \ar[r]^{\Phi_{\g}} & {\sh{S_\g}}}
\end{align}
where $\Phi_\g$ denotes the corresponding map between groupoids for the pinched surface.
At the level of shear coordinates, the map $\pi_\g$  sends the coordinate change $\Phi_{\la\la'}$ on $\T(S)$ to $\Phi_{\la_\g\la'_\g}$ on $\T(S_\g)$. It is a well-defined homomorphism by Theorem~\ref{homom}.

\subsection{The augmented Teichm\"uller space}

We denote by $\mathcal{C}(S)$ the set of multicurves on $S$ together with the empty set representing the ``empty'' multicurve. As a set, the \emph{augmented Teichm\"uller space} \cite{Ber1,Abi1,Abi2} is defined as
\[\overline{\T(S)}=\bigsqcup_{\g\in\mathcal{C}(S)}\T(S_\g)\]
where, if $S_\g$ is disconnected, $\T(S_\g)$ denotes the product of Teichm\"uller spaces of its components. By definition, it contains the Teichm\"uller space $\T(S)=\T(S_{\emptyset})$. The spaces $\T(S_\g)$ for non-empty multicurves $\g$ are called the (lower dimensional) \emph{strata} of $\overline{\T(S)}$. One can also think of elements of $\overline{\T(S)}$ as equivalence classes of $S_\g$--marked hyperbolic surfaces $[\phi_\g\colon S_\g\to\Sigma_\g]$. Once again, we will identify such an element with a hyperbolic metric $m_\g$ on $S_\g$.

This space can be topologized in many equivalent ways, in the realm of complex analysis, group representations or hyperbolic geometry \cite{Abi2}. We will use the following description near a lower dimensional stratum given in terms of lengths of geodesics: first recall that, if $\alpha$ is a simple closed curve on a hyperbolic surface $\Sigma$ with hyperbolic metric $m$ then $\alpha$ is homotopic either to a point, to a puncture or to a unique geodesic $\alpha_m$ for $m$. We let $\ell_m(\alpha)$ be 0 in the first two cases and equal to the length of $\alpha_m$ in the third. Then, given a multicurve $\g$, a sequence $m_n$ in $\T(S)$ converges to $m_\g\in\T(S_\g)$ if, for any simple closed curve $\alpha$ not intersecting $\g$, we have $\ell_{m_n}(\alpha)\to \ell_{m_\g}(\alpha)$. This implies in particular that the lengths of the components of $\g$ tend to 0. Note that in fact one would only need to check the convergence for a sufficiently large fixed collection of curves determining the metric on $S_\g$ (an \
\emph{ample family} in the terminology of \cite{Abi2}).

The action of the mapping class group on the augmented Teichm\"uller space is once again given in terms of change of marking as follows: for $f\in\MCG(S)$ and $[\phi_\g\colon S_\g\to\Sigma_\g]\in\T(S_\g)$, we let
\[f^*[\phi_\g\colon S_\g\to\Sigma_\g]=[\phi_\g\circ f\colon S_{f^{-1}(\g)}\to\Sigma_\g].\]
It is shown in \cite{Abi2} that this action is continuous. It reduces to the standard action on $\T(S)$ if $\g=\emptyset$ and, more generally, for $f\in\MCG(S)$, it induces a map $f^*\colon\T(S_\g)\to\T(S_{f^{-1}(\g)})$ between lower dimensional strata. One of the goals of this section is to understand this map $f^*$ explicitely in terms of shear coordinates.

We can understand this action on the lower dimensional strata further using the exact sequence (\ref{exactseq}) described in Section~\ref{action}. Namely, if $f\in\MCG(S,\g)$, then $f^*m_\g=\zeta(f)^* m_\g$ where the action on the right hand side is that of $\MCG(S_\g)$ on $\T(S_\g)$. In addition, if $T_{\g_i}$ is a Dehn twist along one of the components of $\g$ then $T^*_{\g_i} m_\g=m_\g$. This implies in particular that there are infinite order elements acting non-freely on the augmented Teichm\"uller space. Nonetheless, this action is still relatively well-behaved near the lower dimensional strata and the quotient is topologically the Deligne-Mumford compactification $\overline{\mathcal{M}(S)}$ of the moduli space \cite{Abi2}.


\subsection{Extension of shear coordinates}
\label{extension}

The choice of an ideal triangulation $\la\in\Lambda(S)$ provides a chart $x_\la$ on $\T(S)$, as well as charts on each of the lower-dimensional strata $\T(S_\g)$ associated to the pinched triangulations $\la_\g$. We will denote them by $y_{\la_\g}\colon\T(S_\g)\to\R^{\la_\g}$. In \cite{moi}, we showed that convergence to a lower dimensional stratum in the augmented Teichm\"uller space leads to a certain convergence of the coordinates $x_\la$ to the the coordinates $y_{\la_\g}$. The key ingredient is the following map: given and edge $f\in\la_\g$, recall that $f$ is obtained by taking isotopy classes of  segments of edges in $\la\cap S_\g$. Then, for each edge $e\in\la$, denote by $k(e,f)$ the number of segments of $e$ forming the homotopy class of $f$. We then consider the map
\[\Theta_{\g,\la}\colon\R^\la\to\R^{\la_\g}\]
defined for any $x\colon\la\to\R\in\R^{\la}$ and $f\in\la_\g$ by
\[\Theta_{\g,\la}(x)(f)=\prod_{e\in\la}x(e)^{k(e,f)}\]
In other words, $\Theta_{\g,\la}$ associates to an edge of $\la_\g$ the product of shear parameters along edges of $\la$ which collapse onto $f$.

For a multicurve $\g=\cup_i\g_i$, we let $\ell_{m_t}(\g)=\sum_i\ell_{m_t}(\g_i)$. We then have the following theorem.
\begin{thm}
\label{shear}
Let $m_t\in\T(S)$, $t>0$, be a continuous family of hyperbolic metrics on $S$ and let $m_\g\in\T(S_\g)$ for some multicurve $\g$. If $\lim_{t\to 0}\ell_{m_t}(\g)=0$ then
\[m_t\xrightarrow[t\rightarrow 0]{} m_\g\text{ in }\overline{\T(S)}\Longleftrightarrow \lim_{t\to 0}\Theta_{\g,\la}(x_\la(m_t))=y_{\la_\g}(m_\g)
\]
\end{thm}
\begin{proof}
The forward implication was proven in Proposition~6 of \cite{moi}. We give a proof of the other implication in Appendix~\ref{shearext}.
\end{proof}

This theorem implies that the topology of the augmented Teichm\"uller space can be described explicitely and globally in terms of the shear coordinates associated to a given triangulation $\la$ on $S$. This is in contrast to similar descriptions in terms of Fenchel-Nielsen coordinates. There, given a pants decomposition, that is, a multicurve $\mathcal{P}$ with $3g-3+n$ components, one can easily describe the topology when approaching lower dimensional strata corresponding to pinching some or all of the components of $\mathcal{P}$, by allowing the corresponding lengths coordinates to be 0 and essentially forgetting the twisting coordinates. However, the strata corresponding to pinching curves wich are transverse to $\mathcal{P}$ cannot be easily described in the chart associated to $\mathcal{P}$ and are not undowed with natural induced Fenchel-Nielsen coordinates.

\begin{rem}
The enhanced Teichm\"uller space in which $\T(S)$ embeds can be endowed with a version of the well-known Weil-Petersson Poisson structure (see \cite{fo1}). Its push-forward in the coordinates $x_\la$ has a simple expression in terms of the combinatorics of the triangulation $\la$ which is related to Thurston's intersection form. We showed in \cite{moi} Proposition~9 that this expression is consistant under pinching, so that the logarithm of the map $\Theta_{\g,\la}$ is in fact a Poisson morphism in the sense that, for any $f,f'\in\la_\g$,
\[\left\{\log y_{\la_\g}(f),\log y_{\la_\g}(f')\right\}_{\T(S_\g)}=\left\{\log\Theta_{\g,\la}(x_\la)(f),\log\Theta_{\g,\la}(x_\la)(f')\right\}_{\T(S)}\]
as (constant) functions over $\T(S_\g)$ and $\T(S)$ respectively.
\end{rem}


\subsection{Extension of the changes of coordinates}

Via the map $\pi_\g$, a change of coordinates $\Phi_{\la\la'}$ on $\T(S)$ provides changes of coordinates $\Phi_{\la_\g\la'_\g}$ on each stratum $\T(S_\g)$. We want to show that these changes of coordinates can be understood as ``limits'' of the ones on $\T(S)$.

As such we consider the following diagram:
\begin{equation}
\label{noncom}
\begin{gathered}
\xymatrix{{\R^{\la'}}\ar@<-1ex>[d]_{\Theta_{\g,\la'}} \ar[r]^{\Phi_{\la\la'}} & {\R^\la}\ar@<-1ex>[d]^{\Theta_{\g,\la}}\\
					{\R^{\la'_\g}} \ar[r]_{\Phi_{\la_\g\la'_\g}} & {\R^{\la_\g}}
				 }
\end{gathered}
\end{equation}
It is not in general commutative, however, it becomes so ``at infinity''.
\begin{prop}
\label{behavior}
Suppose $m_t\in\T(S)$, $t>0$, and $m_\g\in\T(S_\g)$ are such that $m_t\to m_\g$ in $\overline{\T(S)}$. Then
\[\lim_{t\to0}\Theta_{\g,\la}\circ\Phi_{\la\la'}(x_{\la'}(m_t))=\lim_{t\to0}\Phi_{\la_\g\la'_\g}\circ\Theta_{\g,\la'}(x_{\la'}(m_t))=y_{\la_\g}(m_\g)
\]
\end{prop}
\begin{proof}
This follows directly from the definitions and Theorem~\ref{shear}. That is, on the one end,
\[\Theta_{\g,\la}\circ\Phi_{\la\la'}(x_{\la'}(m_t))=\Theta_{\g,\la}(x_\la(m_t))\xrightarrow[t\rightarrow 0]{} y_{\la_\g}(m_\g)\]
while, on the other hand,
\[\Phi_{\la_\g\la'_\g}\circ\Theta_{\g,\la'}(x_{\la'}(m_t))\xrightarrow[t\rightarrow 0]{}\Phi_{\la_\g\la'_\g}(y_{\la'_\g}(m_\g))=y_{\la_\g}(m_\g)
\]
\end{proof}


\subsection{Action of the mapping class group on the augmented Teichm\"uller space}

Given any ideal triangulation $\la\in\Lambda(S)$, the homomorphism from the mapping class group to the Ptolemy groupoid described in Section~\ref{action} translates into a homomorphism
\begin{align*}
\Phi_\la\colon\MCG(S)&\to\sh{S}\\
 f&\mapsto\Phi_{f(\la)\la}\,.
\end{align*}
 
Using this homomorphism and the natural map $\I_{f,\la}\colon\R^{f(\la)}\to\R^\la$ induced by the identification of the edges of $\la$ and $f(\la)$ by $f$, we obtain an (anti-)representation
\begin{align*}
\MCG(S)&\to \mathrm{Rat}(\R^\la)\\
f&\mapsto\Phi_{f,\la}=\I_{f,\la}\circ\Phi_{f(\la)\la}
\end{align*}
in the space of rational functions on $\mathbb{\R}^\la$.
This representation coincides with the action of the mapping class group on $\T(S)$ in the sense that
\[\Phi_{f,\la}(x_\la(m))=x_\la(f^*m)\]
since $\I_{f,\la}(x_{f(\la)}(m))=x_\la(f^*m)$. This construction was described first by Penner in \cite{Penner1} using the closely related $\la$--lengths coordinates on the decorated Teichm\"uller space.

We now want to express the action of the mapping class group on the augmented Teichm\"uller space in terms of shear coordinates. Recall that, by definition, if $m_\g\in\T(S_\g)$, then $f^* m_\g\in\T(S_{f^{-1}(\g)})$. As such, we want to fix a triangulation $\la\in\Lambda(S)$ and describe this action as a map between induced charts $\Phi_{f,\la,\g}\colon\R^{\la_\g}\to\R^{\la_{f^{-1}(\g)}}$.

We claim that this map is given by the lower line in Diagram~(\ref{mcgaction}):
\begin{equation}
\label{mcgaction}
\begin{gathered}
\xymatrix{\R^{\la}\ar@<-1ex>[d]_{\Theta_{\g,\la}} \ar[r]^{\Phi_{f(\la)\la}} &{\R^{f(\la)}}\ar@<-1.5ex>[d]_{\Theta_{\g,f(\la)}} \ar[r]^{\I_{f,\la}} & {\R^\la\ \ \ \ \ \ }\ar@<-3ex>[d]^{\Theta_{f^{-1}(\g),\la}}\\
					{\ \R^{\la_\g}} \ar[r]_{\Phi_{f(\la)_\g\la_\g}} &{\ \ \R^{f(\la)_\g}}\ar[r]_{\I_{f,\la,\g}} & {\R^{\la_{f^{-1}(\g)}}}
				 }
\end{gathered}
\end{equation}

In this diagram, the map $\I_{f,\la,\g}$ is defined as follows: $f$ identifies the family of segments $\la\cap S_{f^{-1}(\g)}$ with the family $f(\la)\cap S_\g$. Taking isotopy classes of arcs, we obtain an identification of the edges of $\la_{f^{-1}(\g)}$ with those of $f(\la)_\g$ which determines the map $\I_{f,\la,\g}$. By construction, the second square in Diagram~(\ref{mcgaction}) is commutative.

\begin{thm}
The action of the mapping class group on $\overline{\T(S)}$ is given in shear coordinates by the collection of maps
\[\Phi_{f,\la,\g}=\I_{f,\la,\g}\circ\Phi_{f(\la)_\g\la_\g}\colon\R^{\la_\g}\longrightarrow\R^{\la_{f^{-1}(\g)}}\]
for all $\g\in\mathcal{C}$, in the sense that
\[\Phi_{f,\la,\g}(y_{\la_\g}(m_\g))=y_{\la_{f^{-1}(\g)}}(f^*m_\g)\]

In addition, if $f\in\MCG(S,\g)$ then
\[\Phi_{f,\la,\g}=\Phi_{\zeta(f),\la_\g}\]
\end{thm}
\begin{proof}
Let $m_\g\in\T(S_\g)$ and consider a continuous family $m_t\in\T(S)$, $t>0$, such that $m_t\to m_\g$ as $t\to 0$. Note that, by continuity of the action, we have that $f^*m_t\to f^* m_\g$.

By Theorem~\ref{shear}, we have $\Theta_{\g,\la}(x_\la(m_t))\to y_{\la_\g}(m_\g)$, therefore it suffices to show that $\Phi_{\la,f,\g}\circ\Theta_{\g,\la}(x_\la(m_t))\to y_{\la_{f^{-1}(\g)}}(f^*m_\g)$.

This, in turn, follows from the following equalities:
\begin{align*}
\lim_{t\to 0}\Phi_{\la,f,\g}\circ\Theta_{\g,\la}(x_\la(m_t))&=\lim_{t\to 0}\I_{f,\la,\g}\circ\Theta_{\g,f(\la)}\circ\Phi_{f(\la)\la}(x_\la(m_t))\\
            &=\lim_{t\to 0}\Theta_{f^{-1}(\g),\la}\circ\I_{f,\la}\circ\Phi_{f(\la)\la}(x_\la(m_t))\\
            &=\lim_{t\to 0}\Theta_{f^{-1}(\g),\la}(x_\la(f^*m_t))\\
            &=y_{\la_{f^{-1}(\g)}}(f^*m_\g).
\end{align*}
Here the first equality follows from Proposition~\ref{behavior}, the second from the commutativity of the second square in Diagram~(\ref{mcgaction}), the third from the expression of the mapping class group action in shear coordinates and the last from Proposition~\ref{behavior} again together with the continuity of the action.

If $f\in\MCG(S,\g)$, we saw in the proof of Theorem~\ref{tower} that $f(\la)_\g$ and $\zeta(f)(\la_\g)$ are isotopic. As such, we have $\Phi_{f(\la)_\g\la_\g}=\Phi_{\zeta(f)(\la_\g)\la_\g}$
and, since $f^{-1}(\g)=\g$, $\I_{f,\la,\g}=\I_{\zeta(f),\la_\g}$,
which proves the second part of the theorem.
\end{proof}

We describe some of the maps involved in this theorem in Appendix~\ref{example}, in the case of the once-punctured torus.


\appendix
\section{An example: the once punctured torus}
\label{example}
As an illustration, we write down explicitely some of the maps involved in Diagram~(\ref{mcgaction}) and check the results of Section~\ref{section 2} in the case of the once-punctured torus $S=S_{1,1}$. Consider the triangulation $\la$ of $S$ given in the upper left corner of Figure~\ref{torus1}, where the horizontal and vertical sides are identified as usual. in the middle is its image after a Dehn twist $f=T_\g$ along the curve $\g$ crossing $S$ vertically, and to the right is the original triangulation back. Below is the induced triangulation $\la_\g=f(\la)_\g=\la_{f^{-1}(\g)}$ of the thrice-punctured sphere $S_\g$. The starred punctures correspond to the ones obtained by pinching $\g$.

\begin{figure}[htb!]
\includegraphics{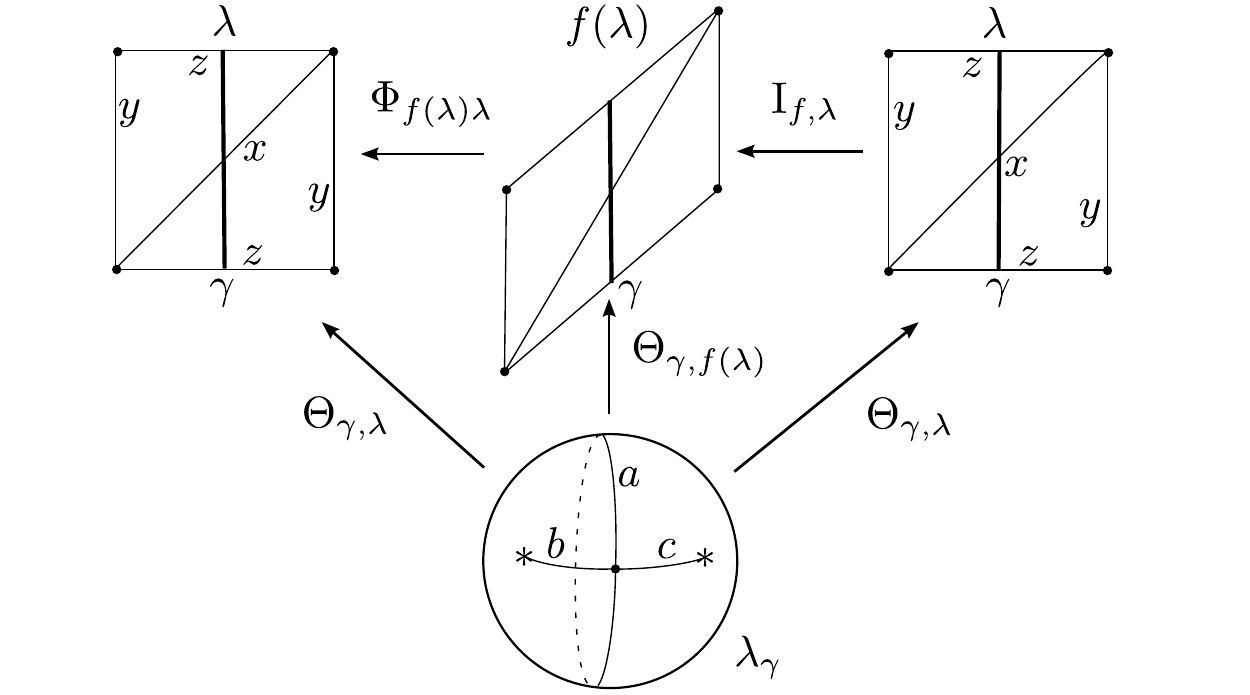}
\caption{\label{torus1} Pinching and twisting along the same curve}
\end{figure}

We write the maps involved in terms of rational functions between fields of fractions, and we refer once again to \cite{Liu1} for the formulas used here. Note, however, that all the maps involved have to go in the opposite direction. Using the variables associated to the edges as on Figure~\ref{torus1}, we obtain:
\begin{align*}
 \Phi_{f,\la}=\Phi_{f(\la)\la}\circ\I_{f,\la}\colon x&\mapsto z^{-1}           & \Theta_{\g,\la}\colon a &\mapsto y\\
                    y&\mapsto (1+z^2)y         &                       b &\mapsto xz\\
                    z&\mapsto (1+z^{-2})^{-1}x &                       c &\mapsto xz
\end{align*}

If we plug in a family of hyperbolic metrics $m_t$ with shear paremeters $(x_t,y_t,z_t)$ converging as $t\to 0$ to the pinched metric $m_\g$  with parameters $(a,b,c)$ (which are $(1,1,1)$ in the case of cusp ends), Theorem~\ref{shear} implies that $y_t\to a$ and $x_t z_t \to b=c$. In addition $z_t\to 0$ necessarily since $l_{m_t}(\g)\to 0$ (see Lemma~\ref{pattern} in Appendix~\ref{shearext}). As such, we indeed have $\lim_{t\to 0}\Phi_{f,\la}\circ\Theta_{\g,\la}(m_t)=\Theta_{\g,\la}(m_\g)$ in this case and $\Phi_{f,\la,\g}=\id$ as it should be. Notice, however, that the commutativity of the diagram is only achieved ``at infinity'' in this case.

If one pinches along the curve $\alpha$ going once horizontally however, the triangulations $\la_\alpha$ and $\la_{f^{-1}(\alpha)}$ are different, as shown in Figure~\ref{torus2}. In this picture the starred punctures correspond to the pinching of $\alpha$.

\begin{figure}[htb!]
\includegraphics{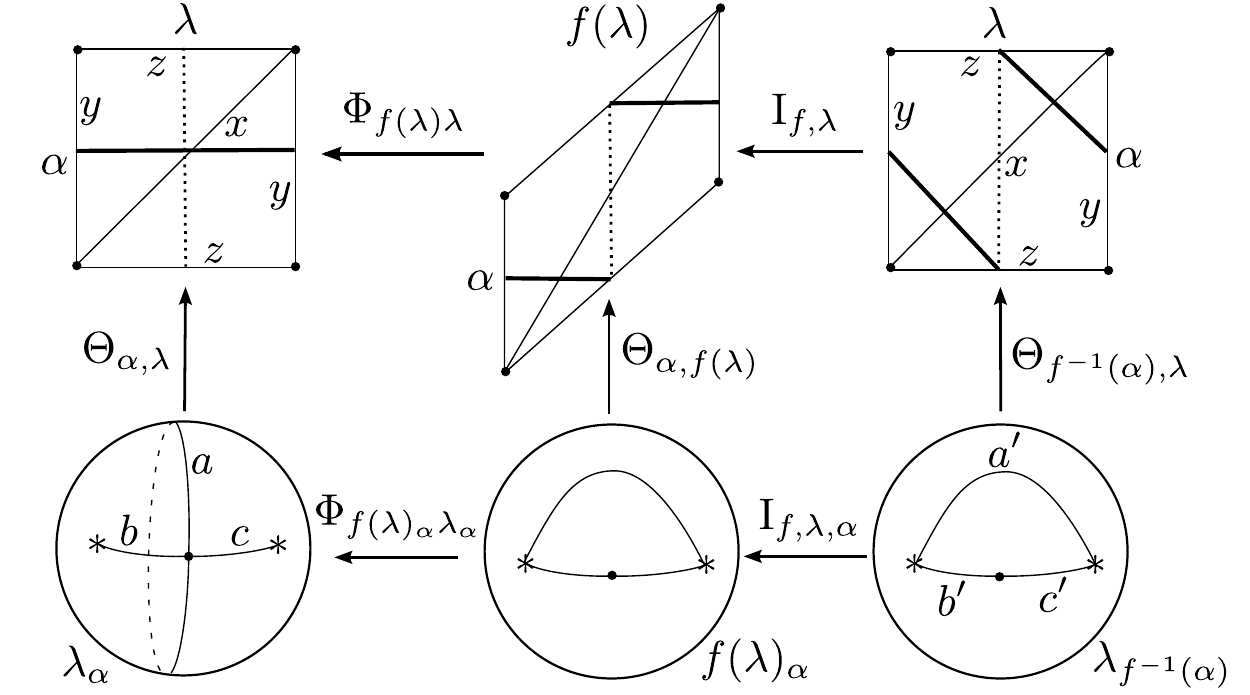}
\caption{\label{torus2} Pinching and twisting along intersecting curves}
\end{figure}

It is important to notice that, although the two pictures in the lower right corner are identical, they correspond to two different embeddings of the thrice-punctured sphere in $S$. There is still a natural identification of the edges of the induced triagulations, however, which gives the map $\I_{f,\la,\g}$. In this case, we have

\begin{align*}
 \Theta_{\alpha,\la}\colon a & \mapsto z & \Theta_{f^{-1}(\alpha),\la}\colon a' &\mapsto x       
 & \Phi_{f,\la,\alpha}\colon a'&\mapsto a^{-1}\\
                    b&\mapsto xy   &  b' &\mapsto xyz & b'&\mapsto ab\\
                    c & \mapsto xy & c'&\mapsto xyz & c'&\mapsto ac
\end{align*}

Note that here the diagram is in fact commutative:
\[\Phi_{f,\la}\circ\Theta_{f^{-1}(\alpha),\la}=\Theta_{\alpha,\la}\circ\Phi_{f,\la,\alpha}\]

\section{Extension of shear coordinates}
\label{shearext}

We give a proof of the reverse implication of Theorem~\ref{shear}. This is a direct consequence of  the following proposition. We recall that, for a multicurve $\g=\cup_i\g_i$, we let $\ell_{m_t}(\g)=\sum_i\ell_{m_t}(\g_i)$.

\begin{prop}
\label{topology}
Let $m_t\in\T(S)$, $t>0$, be a continuous family and $m_\g\in\T(S_\g)$ for some multicurve $\g$ be such that
\[\lim_{t\to 0}\ell_{m_t}(\g)=0\]
and
\[\lim_{t\to 0}\Theta_{\g,\la}(x_\la(m_t))=y_{\la_\g}(m_\g)\]
for a given triangulation $\la\in\Lambda(S)$.

Then, for any simple closed curve $\alpha$ on $S$ such that $\alpha\cap\g=\emptyset$, we have
\[\lim_{t\to 0}\ell_{m_t}(\alpha)=\ell_{m_\g}(\alpha)\]
\end{prop}

In order to prove this proposition we will need to introduce a couple of definitions and state an intermediate lemma.

Shear parameters are defined in terms of edges or, equivalently, in terms of pairs of adjacent triangles. This can be extended naturally to any pair of triangles as follows \cite{Bo1}: Fix a hyperbolic metric $m\in\T(S)$ and an ideal triangulation $\la$ of $S$ realized by geodesic arcs for $m$. Let $P$ and $Q$ be two triangles for $\la$ and let $\kappa$ be a simple (geodesic) arc from a point in $P$ to one in $Q$. The \emph{shearing cocycle} $\s_m(P,Q;\kappa)$ for $m$ associated to the triple $(P,Q;\kappa)$ is given by the sum of the logarithmic shear parameters associated to the edges crossed by $\kappa$ counted with multiplicity. This quantity is of course invariant under homotopy of $\kappa$ respecting the crossing pattern with $\la$.

We now describe how shearing cocycles diverge under pinching of a simple closed curve $\g$. Suppose that $\g$ crosses two triangles $P$ and $Q$ (not necessarily consecutively) and let $\kappa_\g$ be a segment of $\g$ going from $P$ to $Q$. We fix the ends of $\kappa_\g$ on the sides where $\g$ enters the first triangle and exits the second (for either orientation). Let $\wt{\kappa}_\g$ be a lift of $\kappa_\g$ to the universal cover connecting lifts $\wt{P}$, $\wt{Q}$ of $P$ and $Q$. We say that  $\kappa_\g$ \emph{separates} $P$ and $Q$ if their lifts are in one of the configurations of Figure~\ref{zigzag}. The two configurations differ by the pattern in which $\wt{\kappa}_\g$ crosses $\wt{P}$ and $\wt{Q}$. In the first case, for an arbitrary orientation, $\wt{\kappa}_\g$ does a left turn when crossing the first triangle and a right turn when crossing the second one, while in the second case the curve does a right turn first followed by a left turn. We say that $\kappa_\g$ crosses $P$ and $Q$ in a \emph{left-right} and \emph{right-left pattern} respectively.

\begin{figure}[htb!]
\includegraphics{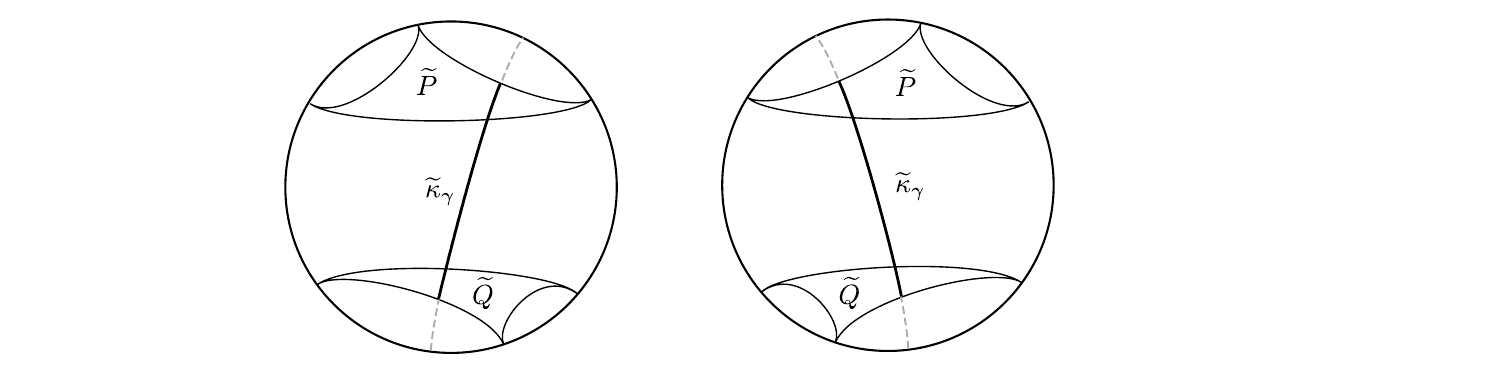}
\caption{\label{zigzag}a left-right and right-left pattern}
\end{figure}

The following lemma is an easy exercise in hyperbolic geometry. Note that we do not assume convergence in the augmented Teichm\"uller space in the statement.
\begin{lem}
\label{pattern}
With the notations above, suppose that $\kappa_\g$ separates the two triangles $P$ and $Q$ and let $m_t$ be a continuous family of hyperbolic metrics on $S$. Then, if $\ell_{m_t}(\g)\to 0$ as $t\to 0$, we have
\begin{itemize}
\item $\sigma_{m_t}(P,Q;\kappa_\g)\to+\infty$ if $\kappa_\g$ crosses the triangles in a left-right pattern;
\item $\sigma_{m_t}(P,Q;\kappa_\g)\to-\infty$ if $\kappa_\g$ crosses the triangles in a right-left pattern.
\end{itemize}
\end{lem}

To prove the proposition we will instead show the convergence of the trace functions $T_m(\alpha)=2\cosh(\ell_m(\alpha)/2)$. We start by introducing a convenient way of describing trace functions in terms of shear coordinates \cite{francis}. Fix an ideal triangulation $\lambda$ of $S$ and let $\alpha$ be a simple closed curve which intersects $\lambda$ transversely and does not cross any edge twice in a row. We endow $\alpha$ with an arbitrary orientation. Then, in each triangle crossed, $\alpha$ does either a left turn or a right turn. If $x$ is the shear parameter associated to the edge through which $\alpha$ enters the triangle, we consider the following two elementary labeled diagrams $L_x$ and $R_x$, depending on whether $\alpha$ turns to the left or to the right.
\begin{align*}
L_x:\ \vcenter{\xymatrix{{x^{1/2}}\ar[r]\ar[dr]&{*}\\
				{x^{-1/2}}\ar[r]&{*}}}
				&\hspace{2cm} R_x:\ 
				\vcenter{\xymatrix{{x^{1/2}}\ar[r]&{*}\\
				{x^{-1/2}}\ar[r]\ar[ur]&{*}}}
\end{align*}
We call such diagrams \emph{dominoes}.

Picking a basepoint on $\alpha$ inside one of the triangles of $\la$, we go along the curve once for the given orientation and paste together the dominoes to obtain a diagram of the following form:

\begin{align*}
D(\alpha,\la):\ \vcenter{\xymatrix{{x^{1/2}_{i_1}}\ar[r]\ar[dr]&{x^{1/2}_{i_2}}\ar[r]&{\cdots}\ar[r]&{x^{1/2}_{i_n}}\ar[r]&{x^{1/2}_{i_1}}\\				{x^{-1/2}_{i_1}}\ar[r]&{x^{-1/2}_{i_2}}\ar[ur]\ar[r]&{\cdots}\ar[r]\ar[ur]&{x^{-1/2}_{i_n}}\ar[r]\ar[ur]&{x^{-1/2}_{i_1}}
}}
\end{align*}
where we identify the first and last columns. We call such a diagram $D(\alpha,\la)$ a \emph{domino diagram}. An \emph{admissible path} in the diagram is a sequence of consecutive edges which starts and ends at the same point in the first and last row. To each such path we associate the monomial given by taking the product of labelings along the path. Using the description of the monodromy representation associated to $m$ in terms of left and right matrices as in \cite{BoWo} Lemma~3, it follows readily that the trace $T_m(\alpha)$ is given by taking the sum of the monomials associated to all the admissible paths in $D(\alpha,\la)$.

\begin{proof}[Proof of Proposition~\ref{topology}]
We want to understand how the domino diagram  $D(\alpha,\la)$ reduces to the diagram $D(\alpha,\la_\g)$ when $t\to 0$. In fact, it suffices to consider the subdiagrams $D^i_\g$ of $D(\alpha,\la_\g)$ associated to two consecutive triangles separated by an edge $\mu_i$ and the corresponding subdiagram $D^i$ of $D(\alpha,\la)$ given by the triangles separated by the edges $\la_{j_1},\ldots,\la_{j_n}$ of $\la$ which collapse onto $\mu_i$.
\begin{align*}
D^i:\ \vcenter{\xymatrix{{}\ar@{.>}[r]\ar@{.>}[dr]&{x^{1/2}_{j_1}}\ar[r]&{\cdots}\ar[r]&{x^{1/2}_{j_n}}\ar@{.>}[r]&{}\\				{}\ar@{.>}[r]&{x^{-1/2}_{j_1}}\ar[ur]\ar[r]&{\cdots}\ar[r]\ar[ur]&{x^{-1/2}_{j_n}}\ar@{.>}[r]\ar@{.>}[ur]&{}}}&\ \mathop{\longrightarrow}^{t\to 0}_?\ D^i_\g:\ \vcenter{\xymatrix{{}\ar@{.>}[r]\ar@{.>}[dr]&{y^{1/2}_i}\ar@{.>}[r]&{}\\				{}\ar@{.>}[r]&{y^{-1/2}_i}\ar@{.>}[r]\ar@{.>}[ur]&{}}}
\end{align*}
First notice we can assume that $\g$ contains a simple closed curve around each of the punctures of $S$. By assumption, the length of each of these components is $0$ for $m_t$ so that the condition $l_{m_t}(\g)\to 0$ is preserved. In this case, the two triangles which are adjacent to $\mu_i$ correspond to two triangles in $\la$ which are crossed by components of $\g$ coming from different sides of one of them, staying parallel for a while then diverging again in the second triangle. the curve $\alpha$ stays between the segments of $\g$ where they are parallel and necessarily diverges from at least one of its components at the beginning and at the end. A possible configuration is illustrated in Figure~\ref{div}.
\begin{figure}[htb!]
\includegraphics{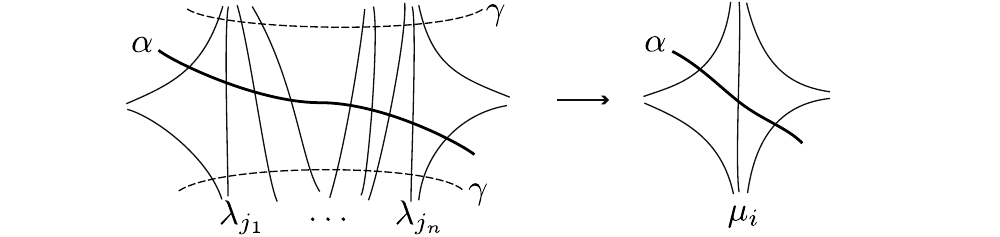}
\caption{\label{div}}
\end{figure}

What we have to show is that, among all the paths in $D^i$ between the first and last columns, the associated monomials always converge to 0, except for the top and bottom paths where they converge to the corresponding labelings of $D^i_\g$. Applying this fact at each edge of $\la_\g$ crossed by $\alpha$, one can easily see that monomials associated to admissible paths for $D(\alpha,\la)$ converge to the corresponding monomials for $D(\alpha,\la_\g)$.

This in turn follows from Lemma~\ref{pattern}. Notice first that the product of labels associated to the top path is exactly $z_t:=[\Theta_{\g,\la}(x_\lambda(m_t))(\mu_i)]^{1/2}$ which converges to $y^{1/2}_i$ by assumption. Similarly, the bottom path corresponds to the term $z^{-1}_t$ and converges to $y^{-1/2}_i$. A general path in $D^i$ which starts on the top row then corresponds to the product of $z_t$ with terms of the form $e^{-\sigma(P_{j_k},P_{j_l};\kappa_\g)}$ for each of the segments of the path which are in the bottom part.
\begin{gather*}
z_t\\
\xymatrix{{x^{1/2}_{j_1}}\ar[r]&{\cdots}\ar[r]&\ar@{.>}[r]\ar[dr]&{x^{1/2}_{j_k}}\ar@{.>}[r]&{\cdots}\ar@{.>}[r]&{x^{1/2}_{j_l}}\ar@{.>}[r]&{\cdots}\ar[r]&{x^{1/2}_{j_n}}\\
{x^{-1/2}_{j_1}}\ar@{.>}[r]&{\cdots}\ar@{.>}[r]&\ar@{.>}[r]&{x^{-1/2}_{j_k}}\ar[r]&{\cdots}\ar[r]&{x^{-1/2}_{j_l}}\ar[ur]\ar@{.>}[r]&{\cdots}\ar@{.>}[r]&{x^{-1/2}_{j_n}}
\save "1,1"."1,8"*[F]\frm{}\restore
\save "2,4"."2,6"*[F]\frm{}\restore
}\\
\ \ \ \ \ \ e^{-1/2\s_{m_t}(P_{j_k},P_{j_l};\kappa_\g)}
\end{gather*}
\begin{figure}[htb!]
\includegraphics{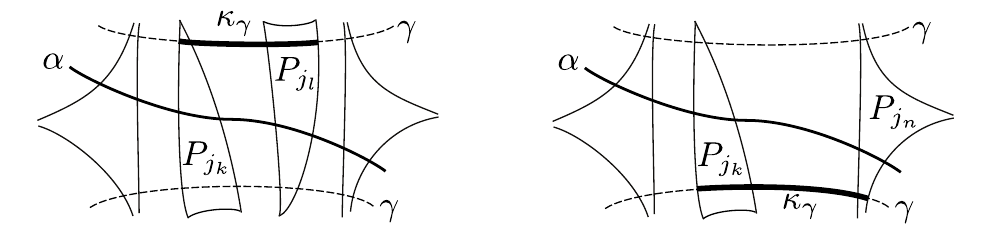}
\caption{\label{divbis}}
\end{figure}
Such a segment of path necessarily starts right after a left domino $L_{x_{j_{k-1}}}$ and ends at a right domino $R_{x_{j_l}}$ or at the bottom entry of the last column of the diagram. Since $\alpha$ and $\gamma$ are parallel between the two triangles, in the case that $P_{j_l}$ is not the last triangle, this implies that the corresponding arc $\kappa_\g$ crosses the triangles in a left-right pattern, as on the left of Figure~\ref{divbis}. In the case that $l=n$, one of the segments of $\gamma$ involved necessarily diverges from $\alpha$, so that $\g$ must contain an arc $\kappa_\g$ crossing the triangles in a left-right pattern. As such, in either case, the monomial associated to a path starting in the top row converges to 0 unless it stays in the top row in which case it converges to $y^{1/2}_i$. Similarly, a path starting in the bottom row will correspond to the product of $z^{-1}_t$ with terms of the form $e^{\sigma(P_{j_k},P_{j_l};\kappa_\g)}$ where $\kappa_\g$ now crosses the triangles in a right-left 
pattern, hence the corresponding monomial converges to 0 unless the path stays in the bottom row in which case the monomial converges to $y^{-1/2}_i$.
\end{proof}

\end{document}